\documentclass[a4paper,12pt]{article}
\usepackage{latexsym}
\usepackage{amsmath}
\usepackage{amssymb}
\usepackage{enumerate}
\usepackage{bm}
\usepackage{mathrsfs}

\usepackage{theorem}
\newtheorem{theorem}{Theorem}[section]
\newtheorem{proposition}[theorem]{Proposition}
\newtheorem{lemma}[theorem]{Lemma}
\newtheorem{corollary}[theorem]{Corollary}

\theorembodyfont{\rmfamily}
\newtheorem{proof}{\textmd{\textit{Proof.}}}

\newtheorem{remark}[theorem]{Remark}

\newtheorem{definition}[theorem]{Definition}

\makeatletter

\@addtoreset{equation}{section}
\makeatother

\newcommand{\qedd}{\hfill \Box}
\newcommand{\ve}{\varepsilon}
\newcommand{\del}{\partial}
\newcommand{\lra}{\longrightarrow}
\newcommand{\e}{\mathrm{e}}

\newcommand{\N}{\ensuremath{\mathbb{N}}}
\newcommand{\Z}{\ensuremath{\mathbb{Z}}}
\newcommand{\R}{\ensuremath{\mathbb{R}}}

\newcommand{\cC}{\ensuremath{\mathcal{C}}}
\newcommand{\cE}{\ensuremath{\mathcal{E}}}

\newcommand{\cL}{\ensuremath{\mathcal{L}}}
\newcommand{\cP}{\ensuremath{\mathcal{P}}}

\newcommand{\fm}{\ensuremath{\mathfrak{m}}}

\newcommand{\sS}{\ensuremath{\mathsf{S}}}

\def\vol{\mathop{\mathrm{vol}}\nolimits}

\def\div{\mathop{\mathrm{div}}\nolimits}

\def\loc{\mathop{\mathrm{loc}}\nolimits}

\def\Ent{\mathop{\mathrm{Ent}}\nolimits}

\def\Var{\mathop{\mathrm{Var}}\nolimits}
\def\Ric{\mathop{\mathrm{Ric}}\nolimits}

\def\CD{\mathop{\mathrm{CD}}\nolimits}

\newcommand{\Grad}{\bm{\nabla}}
\newcommand{\Lap}{\bm{\Delta}}

\newcommand{\rev}[1]{\overleftarrow{#1}}

\title{Some functional inequalities\\ on non-reversible Finsler manifolds}
\author{Shin-ichi Ohta\thanks{Department of Mathematics, Kyoto University,
Kyoto, 606-8502, Japan ({\sf sohta@math.kyoto-u.ac.jp}),
{\it Current address}: Department of Mathematics, Osaka University,
Osaka, 560-0043, Japan ({\sf s.ohta@math.sci.osaka-u.ac.jp});
Supported in part by JSPS Grant-in-Aid for Scientific Research (KAKENHI) 15K04844.}}
\date{}
\begin{document}

\maketitle

\begin{abstract}
We continue our study of geometric analysis on (possibly non-reversible) Finsler manifolds,
based on the Bochner inequality established by the author and Sturm.
Following the approach of the $\Gamma$-calculus \`a la Bakry et al,
we show the dimensional versions of the Poincar\'e--Lichnerowicz inequality,
the logarithmic Sobolev inequality, and the Sobolev inequality.
In the reversible case, these inequalities were obtained
by Cavalletti--Mondino in the framework of curvature-dimension condition
by means of the localization method.
We show that the same (sharp) estimates hold also for non-reversible metrics.
\end{abstract}

\section{Introduction}\label{sc:intro}

The aim of this article is to put forward geometric analysis on
possibly \emph{non-reversible} Finsler manifolds (in the sense of $F(-v) \neq F(v)$)
of Ricci curvature bounded below.
Geometric analysis on spaces of Ricci curvature bounded below
is a classical as well as active area.
The classical Riemannian theory has been generalized
to weighted Riemannian manifolds, linear Markov diffusion semigroups
(the \emph{$\Gamma$-calculus} \`a la Bakry, \'Emery, and others,
see the recent comprehensive book \cite{BGL}),
and to metric measure spaces satisfying Lott, Sturm and Villani's
\emph{curvature-dimension condition} (\cite{StI,StII,LV,Vi}).
In fact, a reinforced version of the curvature-dimension condition,
called the \emph{Riemannian curvature-dimension condition},
is equivalent to the Bochner inequality which is the starting point of the $\Gamma$-calculus
(\cite{AGSbe,EKS}).

The class of Finsler manifolds takes an interesting place in the above picture.
First of all, the distance function on a Finsler manifold can be \emph{asymmetric}
($d(y,x) \neq d(x,y)$ is allowed), thus it is not precisely a metric space in the usual sense.
A non-Riemannian Finsler manifold of weighted Ricci curvature bounded below
satisfies the curvature-dimension condition
(in the naturally extended form to asymmetric distances),
but the Riemannian curvature-dimension condition never holds.
Finally, the natural Finsler Laplacian turns out \emph{nonlinear},
and hence the $\Gamma$-calculus does not directly apply.
Nonetheless,
we investigated the nonlinear heat flow associated with the nonlinear Laplacian in \cite{OShf},
and established the \emph{Bochner inequality} in \cite{OSbw}.
This Bochner inequality has many applications including
gradient estimates (\cite{OSbw,Oisop})
and various eigenvalue estimates (\cite{WX,YH1,YH2,Xi}).

We have developed the \emph{nonlinear $\Gamma$-calculus} approach in \cite{OSbw,Oisop},
the current article could be regarded as a continuation.
It is not always possible to generalize a linear argument
to the Finsler setting (the most important example would be the non-contraction
of heat flow in \cite{OSnc}), however, there are also many positive results.
For example, the aforementioned gradient estimates were shown indeed in this line,
and we proved \emph{Bakry--Ledoux's Gaussian isoperimetric inequality}
under $\Ric_{\infty} \ge K>0$ in the sharp form
(\cite{Oisop}, see \cite{BL} for the original Riemannian result).
The latter is a particular case of the \emph{L\'evy--Gromov type isoperimetric inequality},
which was studied in \cite{Oneedle} by means of the \emph{localization} method.
The localization is another powerful technique reducing
an inequality on a space to those on geodesics (see \cite{Kl,CM1,CM2}),
however, it gives only non-sharp estimates in the non-reversible situation
(see \cite{Oneedle} for details).
We will give sharp functional inequalities even in the non-reversible case,
these generalize results in \cite{CM2} which cover reversible Finsler manifolds
(see also \cite{Pr} for a related work on metric measure spaces
enjoying the \emph{Riemannian curvature-dimension condition}).
Let us also remark that $\Ric_N$ for $N<0$ is not treated in \cite{CM2}.
Our arguments essentially follow the known lines of $\Gamma$-calculus.
There arise, however, some technical difficulties due to the nonlinearity,
while the smoothness of the space sometimes gives additional help.

The organization of the article is as follows.
We briefly review the basics of Finsler geometry in Section~\ref{sc:prel}.
In Section~\ref{sc:Lich}, we prove the \emph{Poincar\'e--Lichnerowicz inequality}
\begin{equation}\label{eq:PL0}
\int_M f^2 \,d\fm -\bigg( \int_M f \,d\fm \bigg)^2
 \le \frac{N-1}{KN} \int_M F^2(\Grad f) \,d\fm
\end{equation}
under $\Ric_N \ge K>0$ for $N \in (-\infty,0) \cup [n,\infty)$,
where $n$ is the dimension of the manifold.
Section~\ref{sc:LSI} is devoted to the \emph{logarithmic Sobolev inequality}
\begin{equation}\label{eq:LS0}
\int_M f \log f \,d\fm \le \frac{N-1}{2KN} \int_M \frac{F^2(\Grad f)}{f} \,d\fm
\end{equation}
under $\Ric_N \ge K>0$ with $N \in [n,\infty)$.
In Section~\ref{sc:Sobo}, we show the \emph{Sobolev inequality}
\[ \frac{\|f\|_{L^p}^2 -\|f\|_{L^2}^2}{p-2} \le \frac{N-1}{KN} \int_M F^2(\Grad f) \,d\fm \]
for $1 \le p \le 2(N+1)/N$ under $\Ric_N \ge K>0$ with $N \in [n,\infty]$
(see also Remark~\ref{rm:Sobo} for a slight improvement of the admissible range of $p$).
We finally discuss further related problems in Section~\ref{sc:prob}.
We remark that \eqref{eq:PL0} with $N \in [n,\infty]$ and \eqref{eq:LS0} with $N=\infty$
were obtained in \cite{Oint} as applications of the curvature-dimension condition.

\section{Preliminaries for Finsler geometry}\label{sc:prel}

We briefly review the basics of Finsler geometry
(we refer to \cite{BCS,Shlec,SS} for further reading),
and some facts from \cite{Oint,OShf,OSbw}.
Interested readers can consult \cite{Oisop} for more elaborate preliminaries,
as well as related surveys \cite{Oaspm,Ogren,Onlga}.

Throughout the article, let $M$ be a connected, $n$-dimensional $\cC^{\infty}$-manifold
without boundary such that $n \ge 2$.
We also fix an arbitrary positive $\cC^{\infty}$-measure $\fm$ on $M$.

\subsection{Finsler manifolds}\label{ssc:Fmfd}

Given local coordinates $(x^i)_{i=1}^n$ on an open set $U \subset M$,
we will always use the fiber-wise linear coordinates
$(x^i,v^j)_{i,j=1}^n$ of $TU$ such that
\[ v=\sum_{j=1}^n v^j \frac{\del}{\del x^j}\Big|_x \in T_xM, \qquad x \in U. \]

\begin{definition}[Finsler structures]\label{df:Fstr}
We say that a nonnegative function $F:TM \lra [0,\infty)$ is
a \emph{$\cC^{\infty}$-Finsler structure} of $M$ if the following three conditions hold:
\begin{enumerate}[(1)]
\item(\emph{Regularity})
$F$ is $\cC^{\infty}$ on $TM \setminus \bm{0}$,
where $\bm{0}$ stands for the zero section;

\item(\emph{Positive $1$-homogeneity})
It holds $F(cv)=cF(v)$ for all $v \in TM$ and $c>0$;

\item(\emph{Strong convexity})
The $n \times n$ matrix
\begin{equation}\label{eq:gij}
\big( g_{ij}(v) \big)_{i,j=1}^n :=
 \bigg( \frac{1}{2}\frac{\del^2 (F^2)}{\del v^i \del v^j}(v) \bigg)_{i,j=1}^n
\end{equation}
is positive-definite for all $v \in TM \setminus \bm{0}$.
\end{enumerate}
We call such a pair $(M,F)$ a \emph{$\cC^{\infty}$-Finsler manifold}.
If $F(-v)=F(v)$ holds for all $v \in TM$, then we say that $F$ is \emph{reversible}
or \emph{absolutely homogeneous}.
\end{definition}

Define the \emph{dual Minkowski norm} $F^*:T^*M \lra [0,\infty)$ to $F$ by
\[ F^*(\alpha) :=\sup_{v \in T_xM,\, F(v) \le 1} \alpha(v)
 =\sup_{v \in T_xM,\, F(v)=1} \alpha(v) \]
for $\alpha \in T_x^*M$.
It is clear by definition that $\alpha(v) \le F^*(\alpha) F(v)$.
We remark that, however, $\alpha(v) \ge -F^*(\alpha)F(v)$ does not hold in general.
Let us denote by $\cL^*:T^*M \lra TM$ the \emph{Legendre transform}.
Namely, $\cL^*$ is sending $\alpha \in T_x^*M$ to the unique element $v \in T_xM$
such that $F(v)=F^*(\alpha)$ and $\alpha(v)=F^*(\alpha)^2$.
We can write down in coordinates
\[ \cL^*(\alpha)
 =\frac{1}{2} \sum_{j=1}^n \frac{\del[(F^*)^2]}{\del \alpha_j}(\alpha)
 \frac{\del}{\del x^j} \Big|_x, \qquad
 \text{where}\ \alpha=\sum_{j=1}^n \alpha_j dx^j. \]
The map $\cL^*|_{T^*_xM}$ is linear if and only if
$F|_{T_xM}$ comes from an inner product.

For $x,y \in M$, we define the (asymmetric) \emph{distance} from $x$ to $y$ by
\[ d(x,y):=\inf_{\eta} \int_0^1 F\big( \dot{\eta}(t) \big) \,dt, \]
where the infimum is taken over all piecewise $\cC^1$-curves $\eta:[0,1] \lra M$
such that $\eta(0)=x$ and $\eta(1)=y$.
Note that $d(y,x) \neq d(x,y)$ can happen since $F$ is only positively homogeneous.
A $\cC^{\infty}$-curve $\eta$ on $M$ is called a \emph{geodesic}
if it is locally minimizing and has a constant speed with respect to $d$.

Given each $v \in T_xM \setminus \{0\}$, the positive-definite matrix
$(g_{ij}(v))_{i,j=1}^n$ in \eqref{eq:gij} induces
the Riemannian structure $g_v$ of $T_xM$ as
\begin{equation}\label{eq:gv}
g_v\bigg( \sum_{i=1}^n a_i \frac{\del}{\del x^i}\Big|_x,
 \sum_{j=1}^n b_j \frac{\del}{\del x^j}\Big|_x \bigg)
 := \sum_{i,j=1}^n g_{ij}(v) a_i b_j.
\end{equation}
Notice that this definition is coordinate-free, and we have $g_v(v,v)=F^2(v)$.
One can similarly define $g_{\alpha}^*:T_x^* M \times T_x^* M \lra \R$
for $\alpha \in T_x^* M \setminus \{0\}$.

The main tools in this article are the nonlinear Laplacian,
the associated heat flow, the integration by parts and the Bochner inequality.
Thus we will not use coordinate calculations of covariant derivative, etc.,
so that we do not recall them.

For later use we also introduce the following quantity associated with $(M,F)$:
\begin{equation}\label{eq:us}
\sS_F :=\sup_{x \in M} \sup_{v,w \in T_xM \setminus 0} \frac{g_v(w,w)}{F^2(w)}
 =\sup_{x \in M} \sup_{\alpha,\beta \in T^*_xM \setminus 0}
 \frac{F^*(\beta)^2}{g^*_{\alpha}(\beta,\beta)} \,\in [1,\infty].
\end{equation}
Since $g_v(w,w) \le \sS_F F^2(w)$ and $g_v$ is the `Hessian' of $F^2/2$ at $v$,
the constant $\sS_F$ measures the (fiber-wise) concavity of $F^2$
and is called the ($2$-)\emph{uniform smoothness constant} (see \cite{Ouni}).
We remark that $\sS_F=1$ holds if and only if $(M,F)$ is Riemannian.

\subsection{Weighted Ricci curvature}\label{ssc:wRic}

We begin with a useful interpretation of the \emph{Ricci curvature} of $(M,F)$
found in \cite[\S 6.2]{Shlec}.
Given a unit vector $v \in T_xM \cap F^{-1}(1)$,
we extend it to a non-vanishing $\cC^{\infty}$-vector field $V$
on a neighborhood $U$ of $x$ such that every integral curve of $V$ is geodesic,
and consider the Riemannian structure $g_V$ of $U$ induced from \eqref{eq:gv}.
Then the \emph{Finsler} Ricci curvature $\Ric(v)$ of $v$ with respect to $F$ coincides with
the \emph{Riemannian} Ricci curvature of $v$ with respect to $g_V$
(in particular, it is independent of the choice of $V$).

Inspired by the above interpretation and the theory of weighted Ricci curvature
(also called the \emph{Bakry--\'Emery--Ricci curvature}) of Riemannian manifolds,
the \emph{weighted Ricci curvature} for $(M,F,\fm)$ was introduced in \cite{Oint} as follows.

\begin{definition}[Weighted Ricci curvature]\label{df:wRic}
Given a unit vector $v \in T_xM$, let $V$ be a $\cC^{\infty}$-vector field
on a neighborhood $U$ of $x$ as above.
We decompose $\fm$ as $\fm=\e^{-\Psi}\vol_{g_V}$ on $U$,
where $\Psi \in \cC^{\infty}(U)$ and $\vol_{g_V}$ is the volume form of $g_V$.
Denote by $\eta:(-\ve,\ve) \lra M$ the geodesic such that $\dot{\eta}(0)=v$.
Then, for $N \in (-\infty,0) \cup (n,\infty)$, define
\[ \Ric_N(v):=\Ric(v) +(\Psi \circ \eta)''(0) -\frac{(\Psi \circ \eta)'(0)^2}{N-n}. \]
We also define $\Ric_{\infty}(v)$ and $\Ric_n(v)$ as the limits and
set $\Ric_N(cv):=c^2 \Ric_N(v)$ for $c \ge 0$.
\end{definition}

We will denote by $\Ric_N \ge K$, $K \in \R$, the condition
$\Ric_N(v) \ge KF^2(v)$ for all $v \in TM$.
Notice that multiplying a positive constant with $\fm$ does not change $\Ric_N$,
thereby, when $\fm(M)<\infty$, we will normalize $\fm$ so as to satisfy $\fm(M)=1$
without loss of generality.
In the Riemannian case, $\Ric_N$ with $N \in [n,\infty]$ has been well studied,
see \cite{Li,Ba,Qi} among many others.
The study of the negative range $N \in (-\infty,0)$ (and even $N \in (-\infty,1]$)
is more recent, see \cite{KM,Mineg,Miharm,Oneg,Wy,WY}.

It is established in \cite{Oint,Oneg,Oneedle}
(for $N \in [n,\infty]$, $N<0$, $N=0$, respectively) that,
for $K \in \R$, the bound $\Ric_N \ge K$  is equivalent to Lott, Sturm and Villani's
\emph{curvature-dimension condition} $\CD(K,N)$.
This characterization extends the corresponding result on weighted Riemannian manifolds
and has many geometric and analytic applications.

\subsection{Nonlinear Laplacian and heat flow}\label{ssc:heat}

For a differentiable function $f:M \lra \R$, the \emph{gradient vector}
at $x$ is defined as the Legendre transform of the derivative of $f$:
$\Grad f(x):=\cL^*(Df(x)) \in T_xM$.
Define the \emph{divergence}, with respect to the measure $\fm$,
of a measurable vector field $V$ with $F(V) \in L_{\loc}^1(M)$
in the weak form as
\[ \int_M \phi \div_{\fm} V \,d\fm :=-\int_M D\phi(V) \,d\fm \qquad
 \text{for all}\ \phi \in \cC_c^{\infty}(M). \]
Then we define the distributional \emph{Laplacian} of $f \in H^1_{\loc}(M)$ by
$\Lap f:=\div_{\fm}(\Grad f)$, that is,
\[ \int_M \phi\Lap f \,d\fm:=-\int_M D\phi(\Grad f) \,d\fm \qquad
 \text{for all}\ \phi \in \cC_c^{\infty}(M). \]
Since the Legendre transform is nonlinear,
our Laplacian $\Lap$ is a nonlinear operator unless $F$ is Riemannian.

In \cite{OShf,OSbw}, we have studied the associated \emph{nonlinear heat equation}
$\del_t u=\Lap u$.
In order to recall some results in \cite{OShf},
we define the \emph{energy} of $u \in H_{\loc}^1(M)$ by
\[ \cE(u):=\frac{1}{2}\int_M F^2(\Grad u) \,d\fm
 =\frac{1}{2}\int_M F^*(Du)^2 \,d\fm. \]
Define $H^1_0(M)$ as the closure of $\cC_c^{\infty}(M)$
with respect to the (absolutely homogeneous) norm
\[ \|u\|_{H^1}:=\|u\|_{L^2} +\{ \cE(u)+\cE(-u) \}^{1/2}. \]

We can construct global solutions to the heat equation
as gradient curves of the energy functional $\cE$
in the Hilbert space $L^2(M)$.
We summarize the existence and regularity properties established in \cite[\S\S 3, 4]{OShf}
in the next theorem (see also \cite[\S 2]{Oisop}).

\begin{theorem}\label{th:hf}
\begin{enumerate}[{\rm (i)}]
\item
For each initial datum $f \in H^1_0(M)$ and $T>0$,
there exists a unique global solution $u=(u_t)_{t \in [0,T]}$ to the heat equation with $u_0=f$.
Moreover, the distributional Laplacian $\Lap u_t$ is absolutely continuous
with respect to $\fm$ for all $t>0$.

\item
The continuous version of a global solution $u$ enjoys
the $H^2_{\loc}$-regularity in $x$ as well as the $\cC^{1,\alpha}$-regularity in both $t$ and $x$.
Moreover, $\del_t u_t$ lies in $H^1_{\loc}(M) \cap \cC(M)$ for all $t>0$,
and further in $H_0^1(M)$ if $\sS_F <\infty$.
\end{enumerate}
\end{theorem}

We set as usual $u_t:=u(t,\cdot)$.
Notice that the standard elliptic regularity yields that
\[ u \in \cC^{\infty} \big( \{(t,x) \in (0,\infty) \times M \,|\, Du_t(x) \neq 0\} \big). \]
Note also that, by the construction of heat flow as the gradient flow of $\cE$:
\begin{itemize}
\item If $c \le u_0 \le C$ almost everywhere,
then $c \le u_t \le C$ almost everywhere for all $t>0$.

\item $\lim_{t \to \infty} \cE(u_t)=0$.
\end{itemize}

We will use the following equation found in \cite[(4.2)]{OSbw}
(see also \cite[Theorem~3.4]{OShf}).

\begin{lemma}\label{lm:dE/dt}
Let $(u_t)_{t \ge 0}$ be a global solution to the heat equation.
Then
\[ \frac{\del}{\del t}\big[ F^2(\Grad u_t) \big] =2D[\Lap u_t](\Grad u_t) \]
holds almost everywhere for all $t>0$.
\end{lemma}

\subsection{Bochner inequality}\label{ssc:BW}

We finally recall the origin of our estimates, the \emph{Bochner inequality},
established in \cite{OSbw} and slightly generalized in \cite{Oneg,Oisop}.

\begin{theorem}[Bochner inequality]\label{th:BWint}
Assume $\Ric_N \ge K$ for some $K \in \R$ and $N \in (-\infty,0) \cup [n,\infty]$.
Given $f \in H_0^1(M) \cap H^2_{\loc}(M) \cap \cC^1(M)$
such that $\Lap f \in H_0^1(M)$,
we have
\begin{align}
&-\int_M D\phi \bigg( \nabla^{\Grad f} \bigg[ \frac{F^2(\Grad f)}{2} \bigg] \bigg) \,d\fm
 \nonumber\\
&\ge \int_M \phi \bigg\{ D[\Lap f](\Grad f) +K F^2(\Grad f)
 +\frac{(\Lap f)^2}{N} \bigg\} \,d\fm \label{eq:BWint}
\end{align}
for all bounded nonnegative functions $\phi \in H_{\loc}^1(M) \cap L^{\infty}(M)$.
\end{theorem}

Recall that $\Lap f \in H^1_0(M)$ is achieved by solutions to the heat equation
if $\sS_F <\infty$.
The operator $\nabla^{\Grad f}$ in the LHS is a linearization of the gradient
given by
\[ \nabla^{\Grad f}u :=\sum_{i,j=1}^n
 g^{ij}(\Grad f) \frac{\del u}{\del x^j} \frac{\del}{\del x^i}, \]
where $(g^{ij}(\Grad f))$ is the inverse matrix of $(g_{ij}(\Grad f))$.
Let us similarly define $\Delta\!^{\Grad f}u:=\div_{\fm}(\nabla^{\Grad f} u)$.
Note that $\nabla^{\Grad f}f=\Grad f$ and $\Delta\!^{\Grad f}f=\Lap f$ hold.

We will sometimes use the following common notation for brevity:
\begin{equation}\label{eq:G2}
\Gamma_2(f) :=\Delta\!^{\Grad f} \bigg[ \frac{F^2(\Grad f)}{2} \bigg]
 -D[\Lap f](\Grad f).
\end{equation}
Then \eqref{eq:BWint} is written in a shorthand way as
$\Gamma_2(f) \ge KF^2(\Grad f)+(\Lap f)^2/N$.

\section{Poincar\'e--Lichnerowicz inequality}\label{sc:Lich}

We start with the Poincar\'e--Lichnerowicz inequality under
the curvature bound $\Ric_N \ge K>0$.
The $N \in [n,\infty]$ case was shown in \cite{Oint}
as a consequence of the curvature-dimension condition.
In the Riemannian setting,
the case of $N \in (-\infty,0)$ was shown independently in \cite{KM} and \cite{Oneg}.

For simplicity we will assume that $M$ is compact
(it automatically holds when $N \in [n,\infty)$ and $M$ is complete),
and normalize $\fm$ as $\fm(M)=1$.
Note that the \emph{ergodicity}:
\begin{equation}\label{eq:ergo}
u_t\ \to\ \int_M u_0 \,d\fm \quad \text{in}\ L^2(M)
\end{equation}
holds for all global solutions $(u_t)_{t \ge 0}$ to the heat equation
(since $\lim_{t \to \infty}\cE(u_t)=0$).

\begin{proposition}\label{pr:3.3.17}
Assume that $M$ is compact and satisfies
$\Ric_N \ge K$ for some $K \in \R$ and $N \in (-\infty,0) \cup [n,\infty]$.
Then we have, for any $f \in H^2(M) \cap \cC^1(M)$
such that $\Lap f \in H^1(M)$,
\[ \int_M \bigg\{ D[\Lap f](\Grad f) +K F^2(\Grad f)
 +\frac{(\Lap f)^2}{N} \bigg\} \,d\fm \le 0. \]
In particular, if $K>0$, then we have
\begin{equation}\label{eq:Lich-}
\int_M F^2(\Grad f) \,d\fm \le \frac{N-1}{KN} \int_M (\Lap f)^2 \,d\fm.
\end{equation}
If $N=\infty$, then the coefficient in the RHS of \eqref{eq:Lich-} is read as $1/K$.
\end{proposition}

\begin{proof}
The first assertion is a direct consequence of Theorem~\ref{th:BWint} with $\phi \equiv 1$.
We further observe, by the integration by parts,
\[ K \int_M F^2(\Grad f) \,d\fm +\frac{\| \Lap f \|_{L^2}^2}{N}
 \le -\int_M D[\Lap f](\Grad f) \,d\fm
 =\| \Lap f \|_{L^2}^2. \]
Rearranging this inequality yields \eqref{eq:Lich-} when $K>0$.
$\qedd$
\end{proof}

Now the Poincar\'e--Lichnerowicz inequality is obtained
by a technique similar to \cite[Proposition~4.8.3]{BGL}.
We define the \emph{variance} of $f \in L^2(M)$ (under $\fm(M)=1$) as
\[ \Var_{\fm}(f) :=\int_M f^2 \,d\fm -\bigg( \int_M f \,d\fm \bigg)^2. \]

\begin{theorem}[Poincar\'e--Lichnerowicz inequality]\label{th:Lich}
Let $M$ be compact, $\fm(M)=1$,
and suppose $\Ric_N \ge K>0$ for some $N \in (-\infty,0) \cup [n,\infty]$.
Then we have, for any $f \in H^1(M)$,
\[ \int_M f^2 \,d\fm -\bigg( \int_M f \,d\fm \bigg)^2
 \le \frac{N-1}{KN} \int_M F^2(\Grad f) \,d\fm. \]
\end{theorem}

\begin{proof}
Let $(u_t)_{t \ge 0}$ be the solution to the heat equation with $u_0=f$,
and put $\Phi(t):=\|u_t\|_{L^2}^2$ for $t \ge 0$.
Observe that
\[ \Phi'(t) =2\int_M u_t \Lap u_t \,d\fm
 =-2\int_M F^2(\Grad u_t) \,d\fm =-4\cE(u_t), \]
and by Lemma~\ref{lm:dE/dt},
\[ \lim_{\delta \downarrow 0}\frac{\Phi'(t+\delta)-\Phi'(t)}{\delta}
 =-4 \lim_{\delta \downarrow 0}\frac{\cE(u_{t+\delta})-\cE(u_t)}{\delta}
 =4\| \Lap u_t \|_{L^2}^2 \]
for all $t>0$.
Then \eqref{eq:Lich-} implies
\begin{equation}\label{eq:Lich0}
-2\Phi'(t) \le \frac{N-1}{KN} \lim_{\delta \downarrow 0}\frac{\Phi'(t+\delta)-\Phi'(t)}{\delta}.
\end{equation}
Notice now that the ergodicity \eqref{eq:ergo} implies
$\lim_{t \to \infty} \Phi(t) =(\int_M f \,d\fm)^2$.
Thus the differential inequality \eqref{eq:Lich0} yields
\[ \Var_{\fm}(f) =-\int_0^{\infty} \Phi'(t) \,dt
 \le \frac{N-1}{2KN} \bigg( \lim_{t \to \infty}  \Phi'(t)-\Phi'(0) \bigg)
 =\frac{2(N-1)}{KN} \cE(f). \]
This completes the proof.
$\qedd$
\end{proof}

\section{Logarithmic Sobolev inequality}\label{sc:LSI}

We next study the logarithmic Sobolev inequality.
From here on we consider only $K>0$ and $N \in [n,\infty)$.
Then $\Ric_N \ge K$ implies the compactness of $M$,
thus we normalize $\fm$ as $\fm(M)=1$ without loss of generality.
We first consider a sufficient condition for the logarithmic Sobolev inequality
as in \cite[Proposition~5.7.3]{BGL}.

\begin{proposition}\label{pr:5.7.3}
Assume that $M$ is compact, $\Ric_{\infty} \ge K>0$, and
\begin{equation}\label{eq:5.7.3}
\int_M \frac{F^2(\Grad u)}{u} \,d\fm
 \le -C \int_M \bigg\{ Du \bigg( \nabla^{\Grad u} \bigg[ \frac{F^2(\Grad u)}{2u^2} \bigg] \bigg)
 +u D[\Lap(\log u)](\Grad[\log u]) \bigg\} \,d\fm
\end{equation}
holds for some constant $C>0$ and all functions $u \in H^2(M) \cap \cC^1(M)$
such that $\Lap u \in H^1(M)$ and $\inf_M u>0$.
Then the \emph{logarithmic Sobolev inequality}
\begin{equation}\label{eq:LSI}
\int_{\{f>0\}} f \log f \,d\fm \le \frac{C}{2} \int_{\{f>0\}} \frac{F^2(\Grad f)}{f} \,d\fm
\end{equation}
holds for all nonnegative functions $f \in H^1(M)$ with $\int_M f \,d\fm=1$.
\end{proposition}

The assumed inequality \eqref{eq:5.7.3} is written in shorthand as (recall \eqref{eq:G2})
\begin{equation}\label{eq:5.7.3'}
\int_M u F^2(\Grad[\log u]) \,d\fm
 \le C \int_M u \Gamma_2(\log u) \,d\fm.
\end{equation}
Note that both sides are well-defined thanks to $\inf_M u>0$
and $F(\Grad u) \in L^{\infty}(M)$.
The LHS of the logarithmic Sobolev inequality is the \emph{relative entropy}
of the probability measure $f\fm$ with respect to $\fm$:
\begin{equation}\label{eq:Ent}
\Ent_{\fm}(f\fm):=\int_M f \log f \,d\fm.
\end{equation}
Since $\fm(M)=1$, we have $\Ent_{\fm}(f\fm) \ge 0$
and $\Ent_{\fm}(f\fm) =0$ holds if and only if $f=1$ almost everywhere.

\begin{proof}
Since $\fm(M)=1$, by the truncation argument,
one can assume $f \ge \ve$ for some $\ve>0$.
Then the solution $(u_t)_{t \ge 0}$ to the heat equation with $u_0=f$
satisfies $u_t \ge \ve$ for all $t>0$, thereby \eqref{eq:5.7.3} admits $u_t$.
Then the claim follows from a similar argument to Theorem~\ref{th:Lich}
applied to $\Psi(t):=\Ent_{\fm}(u_t \fm)$.
We see that
\[ \Psi'(t) =\int_M (\log u_t +1) \del_t u_t \,d\fm
 =-\int_M D[\log u_t](\Grad u_t) \,d\fm
 =-\int_M \frac{F^2(\Grad u_t)}{u_t} \,d\fm. \]
We further deduce from Lemma~\ref{lm:dE/dt} and the integration by parts that
\begin{align*}
\Psi''(t)
&= \int_M \bigg\{ \frac{\Lap u_t}{u_t^2} F^2(\Grad u_t)
 -\frac{2}{u_t}D[\Lap u_t](\Grad u_t) \bigg\} \,d\fm \\
&= -\int_M Du_t \bigg( \nabla^{\Grad u_t} \bigg[ \frac{F^2(\Grad u_t)}{u_t^2} \bigg] \bigg) \,d\fm
 -2\int_M D[\Lap u_t](\Grad[\log u_t]) \,d\fm.
\end{align*}
Since
\[ \int_M D[\Lap u_t](\Grad[\log u_t]) \,d\fm
 = -\int_M \Lap(\log u_t) \Lap u_t \,d\fm
 = \int_M D[\Lap (\log u_t)](\Grad u_t) \,d\fm, \]
the supposed inequality \eqref{eq:5.7.3} means $-2\Psi'(t) \le C \Psi''(t)$.

We observe from the logarithmic Sobolev inequality
under $\Ric_{\infty} \ge K>0$ that
\[ \int_M u_t \log u_t \,d\fm
 \le \frac{1}{2K} \int_M \frac{F^2(\Grad u_t)}{u_t} \,d\fm
 \le \frac{1}{2K\ve} \int_M F^2(\Grad u_t) \,d\fm. \]
Together with $\lim_{t \to \infty} \cE(u_t)=0$,
we have $\lim_{t \to \infty}\Psi(t)=\lim_{t \to \infty}\Psi'(t)=0$ and thus
\[ \int_M f\log f \,d\fm =-\int_0^{\infty} \Psi'(t) \,dt
 \le \frac{C}{2} \int_0^{\infty} \Psi''(t) \,dt =-\frac{C}{2}\Psi'(0). \]
We complete the proof.
$\qedd$
\end{proof}

Now we can show the sharp, dimensional logarithmic Sobolev inequality
along the lines of \cite[Theorem~5.7.4]{BGL}.

\begin{theorem}[Logarithmic Sobolev inequality]\label{th:LSI}
Assume that $\Ric_N \ge K>0$ for some $N \in [n,\infty)$ and $\fm(M)=1$.
Then we have
\[ \int_{\{f>0\}} f \log f \,d\fm \le \frac{N-1}{2KN} \int_{\{f>0\}} \frac{F^2(\Grad f)}{f} \,d\fm \]
for all nonnegative functions $f \in H^1(M)$ with $\int_M f \,d\fm=1$.
\end{theorem}

\begin{proof}
Fix $h \in \cC^{\infty}(M)$ and consider the function $\e^{ah}$ for $a>0$.
We begin with some preliminary and useful equations.
Note first that, since $a>0$,
\begin{equation}\label{eq:e^ah}
\Grad(\e^{ah}) =a \e^{ah} \Grad h, \qquad
 \Lap(\e^{ah}) =a \e^{ah} \{ \Lap h +aF^2(\Grad h) \}.
\end{equation}
Thus, on the one hand,
\begin{align}
\Gamma_2(\e^{ah})
&= \Delta\!^{\Grad h} \bigg[ \frac{a^2 \e^{2ah} F^2(\Grad h)}{2} \bigg]
 -a^2 \e^{ah} D\big[ \e^{ah} \{ \Lap h+a F^2(\Grad h) \} \big] (\Grad h) \nonumber\\
&= a^2 \div_{\fm} \bigg[ \e^{2ah} \nabla^{\Grad h} \bigg[ \frac{F^2(\Grad h)}{2} \bigg]
 +a \e^{2ah} F^2(\Grad h) \Grad h \bigg] \nonumber\\
&\quad -a^2 \e^{2ah} \Big\{ a\{ \Lap h +a F^2(\Grad h) \} F^2(\Grad h)
 +D[\Lap h](\Grad h) +a D[F^2(\Grad h)](\Grad h) \Big\} \nonumber\\
&= a^2 \e^{2ah} \bigg\{ \Delta\!^{\Grad h} \bigg[ \frac{F^2(\Grad h)}{2} \bigg]
 +a D[F^2(\Grad h)](\Grad h) +a^2 F^4(\Grad h) -D[\Lap h](\Grad h) \bigg\} \nonumber\\
&= a^2 \e^{2ah} \big\{ \Gamma_2(h) +a D[F^2(\Grad h)](\Grad h) +a^2 F^4(\Grad h) \big\}.
 \label{eq:G2u}
\end{align}
On the other hand, it follows from the integration by parts that
\begin{align*}
\int_M \Gamma_2(\e^{ah}) \,d\fm &= \| \Lap(\e^{ah}) \|_{L^2}^2
 =a^2 \int_M \e^{2ah} \{ (\Lap h)^2 +2a F^2(\Grad h) \Lap h +a^2 F^4(\Grad h) \} \,d\fm \\
&= a^2 \int_M \e^{2ah} \{ (\Lap h)^2 -2a D[F^2(\Grad h)](\Grad h) -3a^2 F^4(\Grad h) \} \,d\fm.
\end{align*}
Comparing this with \eqref{eq:G2u}, we have
\begin{equation}\label{eq:5.7.8}
\int_M \e^{2ah} (\Lap h)^2 \,d\fm
 =\int_M \e^{2ah} \big\{ \Gamma_2(h) +3a D[F^2(\Grad h)](\Grad h) +4a^2 F^4(\Grad h) \big\} \,d\fm.
\end{equation}

Next we apply the Bochner inequality \eqref{eq:BWint} to $\e^{ah}$
and find, by \eqref{eq:e^ah} and \eqref{eq:G2u},
\begin{align}
&\Gamma_2(h) +a D[F^2(\Grad h)](\Grad h) +a^2 F^4(\Grad h) \nonumber\\
&\ge KF^2(\Grad h) +\frac{1}{N} \{ (\Lap h)^2 +2a F^2(\Grad h) \Lap h +a^2 F^4(\Grad h) \}.
 \label{eq:LiBo}
\end{align}
To be precise, \eqref{eq:LiBo} holds in the weak sense as in Theorem~\ref{th:BWint},
we will employ $\e^h$ as a test function.
Then the RHS is being, by \eqref{eq:5.7.8} with $a=1/2$ and the integration by parts,
\begin{align*}
&\int_M \e^h \bigg\{ KF^2(\Grad h)
 +\frac{1}{N}\big \{ (\Lap h)^2 +2a F^2(\Grad h) \Lap h +a^2 F^4(\Grad h) \big\} \bigg\} \,d\fm \\
&= \int_M \e^h \bigg\{ KF^2(\Grad h) +\frac{a^2}{N} F^4(\Grad h) \bigg\} \,d\fm \\
&\quad +\frac{1}{N} \int_M \e^h \bigg\{ \Gamma_2(h)
 +\frac{3}{2} D[F^2(\Grad h)](\Grad h) +F^4(\Grad h) \bigg\} \,d\fm \\
&\quad -\frac{2a}{N} \int_M \e^h \big\{ F^4(\Grad h) +D[F^2(\Grad h)](\Grad h) \big\} \,d\fm \\
&= K\int_M \e^h F^2(\Grad h) \,d\fm \\
&\quad +\frac{1}{N} \int_M \e^h \bigg\{ \Gamma_2(h)
 +\bigg( \frac{3}{2} -2a \bigg) D[F^2(\Grad h)](\Grad h)
 +(a-1)^2 F^4(\Grad h) \bigg\} \,d\fm.
\end{align*}
Hence we obtain from \eqref{eq:LiBo} that
\begin{align}
\bigg( 1-\frac{1}{N} \bigg) \int_M \e^h \Gamma_2(h) \,d\fm
&\ge K\int_M \e^h F^2(\Grad h) \,d\fm \nonumber\\
&\quad +\frac{3-2(N+2)a}{2N} \int_M \e^h D[F^2(\Grad h)](\Grad h) \,d\fm \nonumber\\
&\quad +\frac{(a-1)^2-Na^2}{N} \int_M \e^h F^4(\Grad h) \,d\fm. \label{eq:5.7.10}
\end{align}
Choosing $a=3/\{2(N+2)\}>0$ gives that
\begin{align}
\frac{N-1}{N} \int_M \e^h \Gamma_2(h) \,d\fm
&\ge K\int_M \e^h F^2(\Grad h) \,d\fm
 +\frac{(4N-1)(N-1)}{4N(N+2)^2} \int_M \e^h F^4(\Grad h) \,d\fm \nonumber\\
&\ge K\int_M \e^h F^2(\Grad h) \,d\fm.
 \label{eq:LSneg}
\end{align}
This is the desired inequality \eqref{eq:5.7.3}
(recall \eqref{eq:5.7.3'}) for $u=\e^h$ with $C=(N-1)/KN$.
Then $u \in \cC^{\infty}(M)$ and $\inf_M u>0$.
By approximation this implies \eqref{eq:5.7.3} for all $u$ in the required class,
therefore we complete the proof by Proposition~\ref{pr:5.7.3}.
$\qedd$
\end{proof}

\begin{remark}\label{rm:LSneg}
Different from the Poincar\'e--Lichnerowicz inequality in the previous section,
the calculation in the above proof does not admit $N$ being negative.
Precisely, $a<0$ is acceptable in the reversible case,
whereas the last inequality \eqref{eq:LSneg} fails for $N<0$.
In fact, the logarithmic Sobolev inequality of the form \eqref{eq:LSI}
does not hold under $\Ric_N \ge K>0$ with $N<0$.
This is because the model space given in \cite{Mineg}
satisfies only the \emph{exponential concentration}, while \eqref{eq:LSI}
(or the Talagrand inequality below) implies the \emph{normal concentration}
(see \cite{Le} for the theory of concentration of measures).
\end{remark}

As a corollary, we have the dimensional \emph{Talagrand inequality}
on the relation between the \emph{Wasserstein distance} $W_2$
and the relative entropy \eqref{eq:Ent}.
See \cite{Oint,Vi} for the definition of $W_2$ as well as
the Talagrand inequality under $\Ric_{\infty} \ge K>0$.
We will denote by $\cP(M)$ the set of all Borel probability measures on $M$.

\begin{corollary}[Talagrand inequality]\label{cr:Tala}
Assume that $\Ric_N \ge K>0$ for $N \in [n,\infty)$ and $\fm(M)=1$.
Then we have, for all $\mu \in \cP(M)$,
\[ W_2^2(\mu,\fm) \le \frac{2(N-1)}{KN} \Ent_{\fm}(\mu). \]
\end{corollary}

\begin{proof}
This is the well known implication going back to \cite{OV}.
Since our distance is asymmetric, we give an outline along \cite{GL} for completeness.

Note that it is enough to show the claim for $\mu=f \fm$ with $f \in H^1(M)$.
Let $(u_t)_{t \ge 0}$ be the global solution to the heat equation with $u_0=f$,
and put $\mu_t:=u_t \fm \in \cP(M)$ and $\Psi(t):=\Ent_{\fm}(\mu_t)$.
Then, as we saw in the proof of Proposition~\ref{pr:5.7.3},
\[ \Psi'(t) =-\int_M \frac{F^2(\Grad u_t)}{u_t} \,d\fm
 \le -\frac{2KN}{N-1} \Psi(t) \]
by Theorem~\ref{th:LSI}.
We can rewrite this inequality as
\[ \sqrt{-\Psi'(t)} \le -\sqrt{\frac{2(N-1)}{KN}} \big( \sqrt{\Psi} \big)'(t). \]
Arguing as in \cite{GL} (following the lines of \cite{GKO}),
we obtain
\[ \bigg( \lim_{\ve \downarrow 0} \frac{W_2(\mu_t,\mu_{t+\ve})}{\ve} \bigg)^2
 \le \int_M \frac{F^2(\Grad u_t)}{u_t} \,d\fm
 =-\Psi'(t) \]
for almost every $t>0$.
Thus we have, since $\lim_{t \to \infty} W_2(\mu_t,\fm) =\lim_{t \to \infty} \Psi(t) =0$,
\begin{align*}
W_2(\mu,\fm) &= \lim_{t \to \infty} W_2(\mu_0,\mu_t)
 \le \int_0^{\infty} \sqrt{-\Psi'(t)} \,dt
 \le -\sqrt{\frac{2(N-1)}{KN}} \int_0^{\infty} \big( \sqrt{\Psi} \big)'(t) \,dt \\
&= \sqrt{\frac{2(N-1)}{KN}} \sqrt{\Psi(0)}.
\end{align*}
This completes the proof.
$\qedd$
\end{proof}

\section{Sobolev inequality}\label{sc:Sobo}

This section is devoted to the Sobolev inequality.
We first derive a non-sharp Sobolev inequality followed by qualitative consequences.
Then, with the help of these qualitative properties, we proceed to the sharp estimate.

\subsection{Non-sharp Sobolev inequality and applications}\label{ssc:logee}

We start with a useful inequality as in \cite[Theorem~6.8.1]{BGL}.

\begin{proposition}[Logarithmic entropy-energy inequality]\label{pr:6.8.1}
Assume $\Ric_N \ge K>0$ for some $N \in [n,\infty)$ and $\fm(M)=1$.
Then we have
\[ \Ent_{\fm}(f^2 \fm)
 \le \frac{N}{2}\log \bigg( 1+\frac{4}{KN} \int_M F^2(\Grad f) \,d\fm \bigg) \]
for all $f \in H^1(M)$ with $\int_M f^2 \,d\fm=1$.
\end{proposition}

\begin{proof}
Let us first consider nonnegative $f$.
By truncation we can assume $f \in L^{\infty}(M)$ and $\inf_M f>0$,
in particular $f^2 \in H^1(M)$.
Consider the solution $(u_t)_{t \ge 0}$ to the heat equation with $u_0=f^2$,
and put $\Psi(t) :=\Ent_{\fm}(u_t \fm)$.
Then we have, by Proposition~\ref{pr:5.7.3},
\[ \Psi'(t) =-\int_M \frac{F^2(\Grad u_t)}{u_t} \,d\fm
 =-\int_M u_t F^2\big( \Grad[\log u_t] \big) \,d\fm
 \le 0. \]
Moreover, recall also from the proof of Proposition~\ref{pr:5.7.3} that
\[ \Psi''(t)
 = -\int_M Du_t \Big( \nabla^{\Grad u_t} \big[ F^2(\Grad[\log u_t]) \big] \Big) \,d\fm
 -2\int_M u_t D(\Lap[\log u_t])(\Grad[\log u_t]) \,d\fm. \]
Then the Bochner inequality \eqref{eq:BWint} shows that
\[ \Psi''(t) \ge -2K\Psi'(t) +\frac{2}{N} \int_M u_t (\Lap[\log u_t])^2 \,d\fm. \]
By the Cauchy--Schwarz inequality and $\int_M u_t \,d\fm=1$, we find
\begin{equation}\label{eq:N>0}
\Psi''(t)
 \ge -2K\Psi'(t) +\frac{2}{N} \bigg( \int_M u_t \cdot \Lap[\log u_t] \,d\fm \bigg)^2
 = -2K\Psi'(t) +\frac{2}{N} \Psi'(t)^2.
\end{equation}
Therefore the function
\[ t\ \longmapsto\ \e^{-2Kt} \bigg( \frac{1}{N} -\frac{K}{\Psi'(t)} \bigg) \]
is nondecreasing, and hence
\[ -\Psi'(t) \le KN \bigg\{ \e^{2Kt} \bigg( 1-\frac{KN}{\Psi'(0)} \bigg) -1 \bigg\}^{-1}. \]
Integrating this inequality gives
\[ \Psi(0) -\Psi(t) \le
 \frac{N}{2} \log\bigg( 1-(1-\e^{-2Kt}) \frac{\Psi'(0)}{KN} \bigg). \]
Note finally that
\[ \Psi'(0) =-\int_M \frac{F^2(\Grad (f^2))}{f^2} \,d\fm
 =-4\int_M F^2(\Grad f) \,d\fm \]
since $f \ge 0$.
Thus letting $t \to \infty$ completes the proof for $f \ge 0$.

In general, we divide $f$ into $f_+:=\max\{f,0\}$ and $f_-:=\max\{-f,0\}$,
and apply the above inequality with respect to $F$ and $\rev{F}(v):=F(-v)$, respectively
(see \cite[Corollary~8.4]{Oint} for details).
Then the concavity of the function $\log(1+s)$ for $s \ge 0$ shows the claim.
$\qedd$
\end{proof}

\begin{remark}\label{rm:Sobneg}
The inequality in \eqref{eq:N>0} is invalid for $N<0$
since the Cauchy--Schwarz inequality cannot be reversed.
\end{remark}

Next we follow the strategy in \cite[Proposition~6.2.3]{BGL}
to show the \emph{Nash inequality} and then a non-sharp Sobolev inequality.

\begin{lemma}[Nash inequality]\label{lm:Nash}
Assume that $\Ric_N \ge K>0$ for some $N \in [n,\infty)$ and $\fm(M)=1$.
Then we have, for all $f \in H^1(M)$,
\[ \|f\|_{L^2}^{N+2} \le
 \bigg( \|f\|_{L^2}^2 +\frac{4}{KN} \cE(f) \bigg)^{N/2} \|f\|_{L^1}^2. \]
\end{lemma}

\begin{proof}
Normalize $f$ so as to satisfy $\int_M f^2 \,d\fm=1$.
Put $\psi(\theta):=\log(\|f\|_{L^{1/\theta}})$ for $\theta \in (0,1]$,
and notice that $\psi$ is a convex function due to the H\"older inequality:
\begin{align*}
&\psi\big( (1-\lambda)\theta+\lambda \theta' \big)
 = \big( (1-\lambda)\theta+\lambda \theta' \big)
 \log\bigg( \int_M |f|^{(1-\lambda)/((1-\lambda)\theta+\lambda \theta')}
 |f|^{\lambda/((1-\lambda)\theta+\lambda \theta')} \,d\fm \bigg) \\
&\le \big( (1-\lambda)\theta+\lambda \theta' \big)
 \log\bigg( \Big( \int_M |f|^{1/\theta} \,d\fm \Big)^{(1-\lambda)\theta/((1-\lambda)\theta+\lambda \theta')}
 \Big( \int_M |f|^{1/\theta'} \,d\fm \Big)^{\lambda \theta'/((1-\lambda)\theta+\lambda \theta')} \bigg) \\
&= \log(\|f\|_{L^{1/\theta}}^{1-\lambda} \cdot \|f\|_{L^{1/\theta'}}^{\lambda})
 =(1-\lambda)\psi(\theta) +\lambda \psi(\theta')
\end{align*}
for all $\theta,\theta' \in (0,1]$ and $\lambda \in (0,1)$.
Therefore
\[ \psi(1) \ge \psi\bigg( \frac{1}{2} \bigg) +\frac{1}{2} \psi'\bigg( \frac{1}{2} \bigg)
 =\frac{1}{2} \psi'\bigg( \frac{1}{2} \bigg) \]
since we supposed $\|f\|_{L^2}=1$.
Combining this with
\[ \psi'\bigg( \frac{1}{2} \bigg)
 =\frac{1}{2} \int_M (-4 \log|f| \cdot |f|^2) \,d\fm
 =-\Ent_{\fm}(f^2 \fm) \]
and Proposition~\ref{pr:6.8.1}, we obtain
\[ \|f\|_{L^1} \ge \exp\bigg( {-}\frac{1}{2}\Ent_{\fm}(f^2 \fm) \bigg)
 \ge \bigg( 1+\frac{8}{KN} \cE(f) \bigg)^{-N/4}
 \ge \bigg( 1+\frac{4}{KN} \cE(f) \bigg)^{-N/2}. \]
This completes the proof.
$\qedd$
\end{proof}

\begin{proposition}[Non-sharp Sobolev inequality]\label{pr:6.2.3}
Assume that $\Ric_N \ge K>0$ for some $N \in [n,\infty) \cap (2,\infty)$ and $\fm(M)=1$.
Then we have
\[ \|f\|_{L^p}^2 \le C_1 \|f\|_{L^2}^2 +C_2 \cE(f) \]
for all $f \in H^1(M)$,
where $p=2N/(N-2)$, $C_1=C_1(N)>1$ and $C_2=C_2(K,N)>0$.
\end{proposition}

\begin{proof}
By the same reasoning as Proposition~\ref{pr:6.8.1} and
\begin{equation}\label{eq:l^p}
\| f_+ \|_{L^p}^2 +\| f_- \|_{L^p}^2 \ge (\| f_+ \|_{L^p}^p +\| f_- \|_{L^p}^p)^{2/p},
\end{equation}
we can assume that $f \in L^{\infty}(M)$ and $\inf_M f>0$.
For $k \in \Z$ consider the decreasing sequence $A_k:=\{f > 2^k\}$
and set
\[ f_k:=\min\!\big\{ {\max}\{f-2^k,0\},2^k \big\}
 =\left\{ \begin{array}{cl}
 2^k & \text{on}\ A_{k+1}, \\ f-2^k & \text{on}\ A_k \setminus A_{k+1}, \\
 0 & \text{on}\ M \setminus A_k.
 \end{array} \right. \]
Notice that $2^k \chi_{A_{k+1}} \le f_k \le 2^k \chi_{A_k}$
($\chi_A$ denotes the characteristic function of $A$)
and hence
\[ 2^{2k} \fm(A_{k+1}) \le \|f_k\|_{L^2}^2 \le 2^{2k} \fm(A_k),
 \qquad \|f_k\|_{L^1} \le 2^k \fm(A_k). \]
Together with the Nash inequality (Lemma~\ref{lm:Nash}) applied to $f_k$,
we have
\begin{align*}
\big( 2^{2k} \fm(A_{k+1}) \big)^{(N+2)/2}
&\le \|f_k\|_{L^2}^{N+2}
 \le \bigg( \|f_k\|_{L^2}^2 +\frac{4}{KN} \cE(f_k) \bigg)^{N/2} \|f_k\|_{L^1}^2 \\
&\le \bigg( 2^{2k} \fm(A_k) +\frac{4}{KN} \cE(f_k) \bigg)^{N/2} 2^{2k} \fm(A_k)^2.
\end{align*}
Let us rewrite this by using $p=2N/(N-2)$ as
\[ 2^{p(k+1)} \fm(A_{k+1}) \le
 2^p \bigg( 2^{2k} \fm(A_k) +\frac{4}{KN} \cE(f_k) \bigg)^{N/(N+2)}
 \big( 2^{pk} \fm(A_k) \big)^{4/(N+2)}. \]
Combining this with the H\"older inequality implies
\begin{align*}
&\sum_{k \in \Z} 2^{p(k+1)} \fm(A_{k+1}) \\
&\le 2^p \sum_{k \in \Z} \bigg\{
 \bigg( 2^{2k} \fm(A_k) +\frac{4}{KN} \cE(f_k) \bigg)^{N/(N+2)}
 \big( 2^{2pk} \fm(A_k)^2 \big)^{2/(N+2)} \bigg\} \\
&\le 2^p
 \bigg( \sum_{k \in \Z} \bigg\{ 2^{2k} \fm(A_k) +\frac{4}{KN} \cE(f_k) \bigg\} \bigg)^{N/(N+2)}
 \bigg( \sum_{k \in \Z} 2^{2pk} \fm(A_k)^2 \bigg)^{2/(N+2)} \\
&\le 2^p
 \bigg( \sum_{k \in \Z} \bigg\{ 2^{2k} \fm(A_k) +\frac{4}{KN} \cE(f_k) \bigg\} \bigg)^{N/(N+2)}
 \bigg( \sum_{k \in \Z} 2^{pk} \fm(A_k) \bigg)^{4/(N+2)}.
\end{align*}
Hence we have
\begin{equation}\label{eq:nonSob}
\bigg( \sum_{k \in \Z} 2^{pk} \fm(A_k) \bigg)^{(N-2)/(N+2)}
 \le 2^p \bigg( \sum_{k \in \Z} \bigg\{ 2^{2k} \fm(A_k) +\frac{4}{KN} \cE(f_k) \bigg\} \bigg)^{N/(N+2)}.
\end{equation}

On the one hand, we deduce from $\sum_{k \in \Z} \cE(f_k)=\cE(f)$ and
\[ \sum_{k \in \Z} 2^{2k} \fm(A_k)
 =\frac{4}{3} \sum_{k \in \Z} 2^{2k} \{ \fm(A_k)-\fm(A_{k+1}) \}
 \le \frac{4}{3} \|f\|_{L^2}^2 \]
that
\[ \sum_{k \in \Z} \bigg\{ 2^{2k} \fm(A_k) +\frac{4}{KN} \cE(f_k) \bigg\}
 \le \frac{4}{3}\|f\|_{L^2}^2 +\frac{4}{KN}\cE(f). \]
On the other hand, we similarly find
\[ \sum_{k \in \Z} 2^{pk} \fm(A_k)
 =\frac{1}{2^p-1} \sum_{k \in \Z} 2^{p(k+1)} \{ \fm(A_k)-\fm(A_{k+1}) \}
 \ge 2^{-p} \|f\|_{L^p}^p. \]
Substituting these into \eqref{eq:nonSob} yields
\[ (2^{-p} \|f\|_{L^p}^p)^{(N-2)/N}
 \le 2^{p(N+2)/N} \bigg( \frac{4}{3}\|f\|_{L^2}^2 +\frac{4}{KN}\cE(f) \bigg). \]
Recalling $p=2N/(N-2)$, we finally obtain
\[ \|f\|_{L^p}^2 \le
 2^{p(N-2)/N} 2^{p(N+2)/N} \bigg( \frac{4}{3}\|f\|_{L^2}^2 +\frac{4}{KN}\cE(f) \bigg)
 = 2^{4N/(N-2)} \bigg( \frac{4}{3}\|f\|_{L^2}^2 +\frac{4}{KN}\cE(f) \bigg). \]
$\qedd$
\end{proof}

We have several qualitative consequences from the non-sharp
Sobolev inequality in Proposition~\ref{pr:6.2.3}.
In fact, one can reduce these qualitative arguments to the Riemannian case
by virtue of the uniform smoothness \eqref{eq:us}.

\begin{corollary}\label{cr:nonSob}
Assume that $\Ric_N \ge K>0$ for some $N \in [n,\infty) \cap (2,\infty)$ and $\fm(M)=1$.
Then there exists a $\cC^{\infty}$-Riemannian metric $g$ for which
\[ \|f\|_{L^p}^2 \le C_1 \|f\|_{L^2}^2 +C_2 \sS_F \cE^g(f) \]
holds for all $f \in H^1(M)$,
with $p=2N/(N-2)$, $C_1>1$ and $C_2>0$ as in Proposition~$\ref{pr:6.2.3}$.
\end{corollary}

\begin{proof}
Let $\{ U_i \}_{i \in \N}$ be an open cover of $M$,
$V_i$ a non-vanishing $\cC^{\infty}$-vector field on $U_i$,
and $\{ \rho_i \}_{i \in \N}$ a partition of unity subordinate to $\{ U_i \}_{i \in \N}$.
Consider the Riemannian metric $g:=\sum_{i \in \N}\rho_i g_{V_i}$.
For any $f \in H^1(M)$, we have
\[ 2\cE^F(f) =\sum_{i \in \N} \int_{U_i} \rho_i F^*(Df)^2 \,d\fm
 \le \sum_{i \in \N} \int_{U_i} \rho_i \sS_F g_{V_i}^*(Df,Df) \,d\fm
 =2\sS_F \cE^g(f). \]
Combining this with Proposition~\ref{pr:6.2.3}, we complete the proof.
$\qedd$
\end{proof}

\subsection{Sharp Sobolev inequality}\label{ssc:Sobo}

We finally show the sharp Sobolev inequality along the lines of \cite[Theorem~6.8.3]{BGL}.

\begin{theorem}[Sobolev inequality]\label{th:Sobo}
Assume that $\Ric_N \ge K>0$ for some $N \in [n,\infty)$ and $\fm(M)=1$.
Then we have
\[ \frac{\|f\|_{L^p}^2 -\|f\|_{L^2}^2}{p-2} \le \frac{N-1}{KN} \int_M F^2(\Grad f) \,d\fm \]
for all $1 \le p \le 2(N+1)/N$ and $f \in H^1(M)$.
\end{theorem}

The case of $p=2$ is understood as the limit, giving the logarithmic Sobolev inequality
(Theorem~\ref{th:LSI}).
The $p=1$ case amounts to the Poincar\'e--Lichnerowicz inequality
(Theorem~\ref{th:Lich}), whereas only for $N \ge n$.

\begin{proof}
Let us assume $N>2$ for simplicity.
This certainly covers the case of $N=n=2$ just by taking the limit as $N \downarrow 2$.
Notice also that, due to \eqref{eq:l^p} for $p>2$
and its converse for $p<2$,
it suffices to show the claim for nonnegative $f$
by a similar argument to Proposition~\ref{pr:6.8.1}.

Take the smallest possible constant $C>0$ satisfying
\begin{equation}\label{eq:C}
\frac{\|f\|_{L^p}^2 -\|f\|_{L^2}^2}{p-2} \le 2C\cE(f)
\end{equation}
for all nonnegative functions $f \in H^1(M)$.
Our goal is to show $C \le (N-1)/KN$.
Let us suppose that we find a extremal (nonconstant) function $f \ge 0$
enjoying equality in \eqref{eq:C} as well as
$f \in L^{\infty}(M)$ and $\inf_M f >0$,
and normalize it as $\|f\|_{L^p}=1$.
For any $\phi \in \cC^{\infty}(M)$ and $\ve>0$, by the choice of $f$,
\[ \frac{\|f+\ve\phi\|_{L^p}^2-\|f\|_{L^p}^2 -\|f+\ve\phi\|_{L^2}^2 +\|f\|_{L^2}^2}{p-2}
 \le 2C \{ \cE(f+\ve\phi) -\cE(f) \}. \]
Dividing both sides by $\ve$ and letting $\ve \downarrow 0$, we have
\[ \frac{1}{p-2} \int_M \bigg( \frac{2}{p} p f^{p-1} \phi -2f\phi \bigg) \,d\fm
 \le 2C \int_M D\phi(\Grad f) \,d\fm
 =-2C \int_M \phi \Lap f \,d\fm. \]
Since $\phi$ was arbitrary, it follows that ($\Lap f$ is well-defined and)
\begin{equation}\label{eq:6.8.3}
f^{p-1} -f=-C(p-2) \Lap f.
\end{equation}

Put $f=\e^u$.
Then \eqref{eq:6.8.3} is rewritten as
\[ \e^{(p-1)u}-\e^u =-C (p-2) \Lap[\e^u]
 =-C (p-2) \e^u \{ \Lap u+F^2(\Grad u) \}, \]
therefore
\begin{equation}\label{eq:6.8.4}
\e^{(p-2)u} =1-C(p-2) \{ \Lap u+F^2(\Grad u) \}.
\end{equation}
Multiply both sides by $\e^{bu} \Lap u$ with $b \ge 0$
and use the integration by parts to find
\begin{align*}
&(p-2+b) \int_M \e^{(p-2+b)u} F^2(\Grad u) \,d\fm \\
&= b\int_M \e^{bu} F^2(\Grad u) \,d\fm \\
&\quad +C(p-2) \int_M \e^{bu} \big\{ (\Lap u)^2 -b F^4(\Grad u) -D[F^2(\Grad u)](\Grad u) \big\} \,d\fm.
\end{align*}
Recall \eqref{eq:5.7.8} with $a=b/2 \ge 0$:
\begin{equation}\label{eq:5.7.8'}
\int_M \e^{bu} (\Lap u)^2 \,d\fm
 =\int_M \e^{bu} \bigg\{ \Gamma_2(u) +\frac{3b}{2} D[F^2(\Grad u)](\Grad u)
 +b^2 F^4(\Grad u) \bigg\} \,d\fm.
\end{equation}
Thus
\begin{align*}
(p-2+b) \int_M \e^{(p-2+b)u} F^2(\Grad u) \,d\fm
&= b\int_M \e^{bu} F^2(\Grad u) \,d\fm +C(p-2) \int_M \e^{bu} \Gamma_2(u) \,d\fm \\
&\quad +C(p-2) \bigg( \frac{3b}{2}-1 \bigg) \int_M \e^{bu} D[F^2(\Grad u)](\Grad u) \,d\fm \\
&\quad +C(p-2)b(b-1) \int_M \e^{bu} F^4(\Grad u) \,d\fm.
\end{align*}
Substituting \eqref{eq:6.8.4} to $\e^{(p-2)u}$ in the LHS,
we have on the one hand
\begin{align}
C(p-2) \int_M \e^{bu} \Gamma_2(u) \,d\fm
&= (p-2) \int_M \e^{bu} F^2(\Grad u) \,d\fm \nonumber\\
&\quad -C(p-2)(p-2+b) \int_M \e^{bu} F^2(\Grad u) \{ \Lap u +F^2(\Grad u) \} \,d\fm \nonumber\\
&\quad -C(p-2) \bigg( \frac{3b}{2}-1 \bigg) \int_M \e^{bu} D[F^2(\Grad u)](\Grad u) \,d\fm \nonumber\\
&\quad -C(p-2)b(b-1) \int_M \e^{bu} F^4(\Grad u) \,d\fm \nonumber\\
&= (p-2) \int_M \e^{bu} F^2(\Grad u) \,d\fm \nonumber\\
&\quad +C(p-2) \bigg( p-1-\frac{b}{2} \bigg) \int_M \e^{bu} D[F^2(\Grad u)](\Grad u) \,d\fm \nonumber\\
&\quad +C(p-2)^2 (b-1) \int_M \e^{bu} F^4(\Grad u) \,d\fm. \label{eq:6.8.5}
\end{align}
On the other hand,
multiplying the RHS of \eqref{eq:LiBo} by $\e^{bu}$ and integrating it gives,
together with \eqref{eq:5.7.8'},
\begin{align*}
&\int_M \e^{bu} \bigg\{ KF^2(\Grad u) +\frac{a^2}{N} F^4(\Grad u) \bigg\} \,d\fm \\
&\quad +\frac{1}{N} \int_M \e^{bu} \bigg\{ \Gamma_2(u)
 +\frac{3b}{2} D[F^2(\Grad u)](\Grad u) +b^2 F^4(\Grad u) \bigg\} \,d\fm \\
&\quad -\frac{2a}{N} \int_M \e^{bu} \big\{ bF^4(\Grad u) +D[F^2(\Grad u)](\Grad u) \big\} \,d\fm \\
&= K\int_M \e^{bu} F^2(\Grad u) \,d\fm \\
&\quad +\frac{1}{N} \int_M \e^{bu}
 \bigg\{ \Gamma_2(u) +\bigg( \frac{3b}{2}-2a \bigg) D[F^2(\Grad u)](\Grad u)
 +(a-b)^2 F^4(\Grad u) \bigg\} \,d\fm.
\end{align*}
Thus we obtain from \eqref{eq:LiBo} the following variant of \eqref{eq:5.7.10}:
\begin{align}
\bigg( 1-\frac{1}{N} \bigg) \int_M \e^{bu} \Gamma_2(u) \,d\fm
&\ge K \int_M \e^{bu} F^2(\Grad u) \,d\fm \nonumber\\
&\quad +\bigg( \frac{3b-4a}{2N}-a \bigg) \int_M \e^{bu} D[F^2(\Grad u)](\Grad u) \,d\fm \nonumber\\
&\quad +\bigg( \frac{(a-b)^2}{N}-a^2 \bigg) \int_M \e^{bu} F^4(\Grad u) \,d\fm.
 \label{eq:6.8.6}
\end{align}
Comparing the coefficients in \eqref{eq:6.8.5} and \eqref{eq:6.8.6},
we would like to choose $a$ and $b$ enjoying
\begin{equation}\label{eq:a&b}
p-1-\frac{b}{2} =\frac{3b-2(N+2)a}{2(N-1)}, \qquad
 (p-2)(b-1) =\frac{(a-b)^2 -Na^2}{N-1}.
\end{equation}

At this point we need an additional care,
because of the non-reversibility, on the ranges of $a$ and $b$.
As we mentioned, they necessarily satisfy $a \ge 0$ and $b \ge 0$.
We first deduce from the first equation in \eqref{eq:a&b} that
\[ a=\frac{b}{2} -(p-1)\frac{N-1}{N+2}. \]
Substituting this into the latter inequality in \eqref{eq:a&b} yields
\begin{equation}\label{eq:b}
\frac{b^2}{4} +\bigg( \frac{p-1}{N+2}-1 \bigg) b
 -(p-2) +(p-1)^2 \bigg( \frac{N-1}{N+2} \bigg)^2 =0.
\end{equation}
Let us denote the LHS by $h(b)$.
The discriminant of $h$ is given by
\[ \frac{N(p-1)}{N+2} \bigg( \frac{2-N}{N+2}(p-1) +1 \bigg), \]
which vanishes at $p=1,2N/(N-2)$ and is nonnegative for $p \in [1,2N/(N-2)]$.
Note also that the axis of $h(b)$ is positive since
$p \le 2(N+1)/N$ ($\le 2N/(N-2)$) implies
\[ 2\bigg( 1-\frac{p-1}{N+2} \bigg) \ge 2 \bigg( 1-\frac{1}{N} \bigg) >0. \]
Thus we have the positive solution
\[ b_0 \ge 2\bigg( 1-\frac{p-1}{N+2} \bigg) \]
of \eqref{eq:b}, and put
\[ a_0:=\frac{b_0}{2} -(p-1)\frac{N-1}{N+2}. \]
Then we observe from $p \le 2(N+1)/N$ that
\begin{equation}\label{eq:a>0}
a_0 \ge 1-\frac{p-1}{N+2} -(p-1)\frac{N-1}{N+2} \ge 0.
\end{equation}
Plugging above $a_0$ and $b_0$ into \eqref{eq:6.8.5} and \eqref{eq:6.8.6},
we obtain $C \le (N-1)/KN$ as desired.

Finally, since there may not be a good extremal function,
one needs an extra discussion on the approximation procedure.
This step, needing only the non-sharp Sobolev inequality,
can be reduced to the Riemannian case by virtue of Corollary~\ref{cr:nonSob}.
See the latter half of the proof of \cite[Theorem~6.8.3]{BGL} for details.
$\qedd$
\end{proof}

\begin{remark}\label{rm:Sobo}
In \eqref{eq:a>0} we used $p \le 2(N+1)/N$ which is slightly more restrictive
than $p \le 2N/(N-2)$ in \cite{BGL,CM2}.
A more precise estimate gives $h(b_0-2a_0) \le 0$ for
\[ p \in \bigg[ \frac{2(N+1)}{N},
 \frac{7N^2+2N+(N+2)\sqrt{N^2+8N}}{4N(N-1)} \bigg], \]
which means $a_0 \ge 0$ and slightly improves the acceptable range of $p$.
This is, however, still more restrictive than $p \le 2N/(N-2)$.
Indeed, in the extremal case of $p=2N/(N-2)$,
one can explicitly calculate (see \cite[Theorem~6.8.3]{BGL})
\[ b_0=2\bigg( 1-\frac{p-1}{N+2} \bigg) =\frac{2(N-3)}{N-2}, \qquad
 a_0=-\frac{2}{N-2}<0. \]
\end{remark}

Thanks to the smoothness of $M$ and $\fm$,
we have the following corollary.

\begin{corollary}[Sobolev inequality for $N=\infty$]\label{cr:Sobo}
Assume that $M$ is compact and satisfies $\Ric_{\infty} \ge K>0$ and $\fm(M)=1$.
Then we have
\[ \frac{\|f\|_{L^p}^2 -\|f\|_{L^2}^2}{p-2} \le \frac{1}{K} \int_M F^2(\Grad f) \,d\fm \]
for all $1 \le p \le 2$ and $f \in H^1(M)$.
\end{corollary}

\begin{proof}
By the smoothness and the compactness, for any $\ve>0$,
we have $\Ric_{N_{\ve}} \ge K-\ve$ for sufficiently large $N_{\ve}<\infty$.
Then the claim is derived from Theorem~\ref{th:Sobo}
as the limit of $\ve \downarrow 0$.
$\qedd$
\end{proof}

\section{Further problems}\label{sc:prob}

\begin{enumerate}[(A)]
\item
Though we did not pursue that direction in this article,
the \emph{$p$-spectral gap} is also treated in \cite{CM2}.
It is worthwhile to study such a problem on non-reversible Finsler manifolds.
See \cite{YH1} for a related work.

\item
There remain many open problems in the case of $N<0$.
We saw that the Poincar\'e--Lichnerowicz inequality admits $N<0$
and the logarithmic Sobolev inequality does not (recall Remark~\ref{rm:LSneg}).
Nonetheless, we had in \cite{Oneg} certain variants
of the logarithmic Sobolev and Talagrand inequalities
under the \emph{entropic curvature-dimension condition} $\CD^e(K,N)$.
The condition $\CD^e(K,N)$ is seemingly stronger than $\Ric_N \ge K$ when $N<0$,
whereas the precise relation is still unclear.
One of the widely open problems for $N<0$ is a gradient estimate for the heat semigroup.
The model space given in \cite{Mineg} would give a clue to the further study.
\end{enumerate}

{\small

}

\end{document}